\documentclass[11pt,twoside]{article}
\usepackage{mathrsfs}
\usepackage{amsmath,amssymb}
\usepackage{amsfonts}
\usepackage{latexsym,bm}
\usepackage{amsthm}
\usepackage{amssymb,amscd}
\usepackage{amsbsy}
\usepackage{fancyhdr,graphicx}
\usepackage[dvips]{psfrag}
\usepackage{indentfirst}
\usepackage{graphics,color}
\usepackage{epsfig}
\usepackage{subfigure}
\usepackage{xcolor}

\numberwithin{equation}{section}

\newtheorem {theorem} {Theorem}
\newtheorem {proposition} [theorem]{Proposition}

\newtheorem {lemma}  [theorem]{Lemma}

\theoremstyle{remark}

\theoremstyle{definition}

\newtheorem {remark} [theorem]{Remark}

\begin{document}
\setlength{\parindent}{4ex}
\setlength{\parskip}{1ex}
\setlength{\oddsidemargin}{12mm}
\setlength{\evensidemargin}{9mm}

\title
{Abelian integrals and limit cycles for a class of cubic polynomial
vector fields of Lotka-Volterra type with a rational first integral of degree 2}
\begin{figure}[b]
\rule[-0.5ex]{7cm}{0.2pt}\\\footnotesize $^{*}$Corresponding author.
\\E-mail address: cenxiuli2010@163.com (X. Cen),
mcszyl@mail.sysu.edu.cn (Y. Zhao) and haiihuaa@tom.com (H. Liang).\\
\end{figure}
\author
{ {Xiuli Cen$^{a,*}$, Yulin Zhao$^{a}$ and Haihua Liang$^{b}$}\\
{\footnotesize\it $^{a}$ Department of Mathematics, Sun Yat-sen
University, Guangzhou, 510275, P.R.China}\\
{\footnotesize\it $^{b}$ Department of Computer Science, Guangdong Polytechnic Normal University,}\\
{\footnotesize\it Guangzhou, 510665, P.R.China}}
\date{}
\maketitle {\narrower \small \noindent {\bf Abstract\,\,\,} In this
paper, we study the number of limit cycles which bifurcate
from the periodic orbits of cubic polynomial vector fields of Lotka-Volterra
type having a rational first integral of degree 2, under polynomial perturbations of degree $n$.
The analysis is carried out by estimating the number of zeros of the corresponding
Abelian integrals. Moreover, using \emph{Chebyshev
criterion}, we show that the sharp upper bound for the number of zeros of the Abelian integrals defined on each period annulus is 3 for $n=3$.  The simultaneous bifurcation and
distribution of limit cycles for the system with two period
annuli under cubic polynomial perturbations are considered. All configurations $(u,v)$ with $0\leq u, v\leq 3, u+v\leq5$ are realizable.

}


\section{Introduction and statement of the main results}
As is known, in the study of the qualitative theory of real planar differential systems, one of the important
open problems is the determination of limit cycles. The second part of the famous Hilbert's 16th problem, proposed in 1900, asks for an upper bound on Hilbert number $H(n)$ and position of limit cycles for all planar polynomial differential systems of degree $n$, but it is still open even for $n=2$. It is so difficult that the weak form of this problem has been introduced. A classical way to obtain limit cycles is that perturbing the periodic orbits of a center. Let us consider the planar polynomial vector fields $X_{\varepsilon}=X_{0}+\varepsilon Y$, where $0<\varepsilon\ll1$ and $X_{0}=(-H_{y}/R,H_{x}/R)$ is a polynomial vector field having a continuum of periodic orbits. $H$ is a first integral of $X_{0}$ and $R$ is an integrating factor. If $R=1$, we call $X_{0}$ a Hamiltonian vector field, otherwise, we say that it's a non-Hamiltonian integrable vector field. In order to study the periodic orbits of $X_{\varepsilon}$ that remain among all the periodic orbits of $X_{0}$, it is necessary to study the number of zeros of an (generalized) Abelian integral, also known as  the first order Melnikov function, i.e.,
\begin{equation}
M(h)=\displaystyle\oint_{H=h}R\left(Y_{1}(x,y)\mathrm{d}y-Y_{2}(x,y)\mathrm{d}x\right),
\end{equation}
where $\{H=h,h\in(h_{0},h_{1})\}$ are periodic orbits of $X_{0}$, and $Y_{i}(x,y),i=1,2$, are the two components of $Y$.

To the best of our knowledge, many authors have investigated the limit cycles for the quadratic Hamiltonian systems and non-Hamiltonian integrable systems under polynomial perturbations (e.g., \cite{CLY,CLP,GGJ,HI,I,LLLZ,LLLZ1,L,Z} and references therein). However, the studies on cubic and higher degree systems are relatively few (e.g., \cite{AZ,BL1,LLLZ,LLLZ2,LLM,YH,ZZ}).

In this paper, we concern with the number of zeros of Abelian integrals for a class of cubic non-Hamiltonian integrable systems of Lotka-Volterra type with a rational first integral of degree 2, the integrating factor of which consists of $x$ and $y$. In general, it is difficult to study the bifurcation of limit cycles for a non-Hamiltonian integrable systems with integrating factors including  both the variables $x$ and $y$, thus the technique and the results are few (e.g., \cite{AZ,BL1,LLLZ2}).

The authors \cite{CL} first studied the planar cubic polynomial
vector fields of Lotka-Volterra type with a rational first integral
of degree 2 and got 28 non-topologically equivalent phase portraits,
among which there are only 2 cases having at least one center in the finite plane,
as shown in Figure 1 (see Figure 1 of \cite{CL}).

\begin{figure}[h]
\centering \epsfig{file=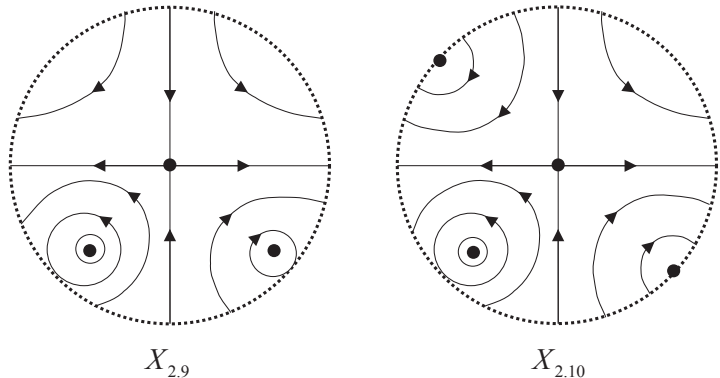,totalheight=50mm,width=95mm,angle=0}
\setlength{\abovecaptionskip}{0pt}%
\renewcommand{\figurename}{Figure}
\caption{All the non-topologically equivalent phase portraits with at least one center
of a planar cubic polynomial vector field of Lotka-Volterra type with a rational first
integral of degree 2.} \centering
\label{Fig.1}
\end{figure}

The system with a center in \cite{CL} is a differential system of the form
\begin{equation}\begin{array}{ll}\label{eq11}
\dot{x}=x(1+bx+x^2-y^2),\\[2ex]
\dot{y}=y(-1-cy+x^2-y^2),
\end{array}\end{equation}
which has a first integral
\begin{equation*}\vspace{.05in}
H=\dfrac{1+bx+cy+x^2+y^2}{xy}
\end{equation*}
with integrating factor $x^{-2}y^{-2}$.

When $0\leq b<2$, $0\leq
c<2$ and $b\neq c$, the phase portrait of system \eqref{eq11} corresponds to
$X_{2.9}$ in Figure 1; when $0<b=c<2$, the phase portrait of system \eqref{eq11}
is $X_{2.10}$ in Figure 1. For convenience, we will also denote the two subclasses of system \eqref{eq11}
by $X_{2.9}$ and $X_{2.10}$, respectively.

System $X_{2.9}$ has three finite singular points: a
hyperbolic saddle at $O (0,0)$ and two centers at $C_{\pm} (\delta(-b\delta\pm
c\gamma)/2(\delta^2-\gamma^2),\gamma(-c\gamma\pm
b\delta)/2(\gamma^2-\delta^2))$, where
$\gamma=\sqrt{4-b^2}, \delta=\sqrt{4-c^2}$. $C_{+}$ is located in the third
quadrant while $C_{-}$ is located in the fourth (resp. second) quadrant if $b<c$ (resp. $b>c$). If $b>c$, taking a change
$(x,y,t,b,c)\rightarrow (y,x,-t,c,b)$, system $X_{2.9}$ is
reduced to that in the case $b<c$. Thus without loss of generality, it suffices to
consider the latter case. $H=h_{\pm}=(-bc\pm\gamma\delta)/2$ correspond to
$C_{\pm}$. There are two families of periodic orbits
$\Gamma_{h}:H=h$, $h\in(-2,h_{-})\cup(h_{+},2)$,
which surround the centers $C_{-}$ and $C_{+}$, respectively.

System $X_{2.10}$ has two finite singular points. $O (0,0)$ is a
hyperbolic saddle and $C (-1/b,-1/b)$ is a center, which is located in the
third quadrant. There is a family of periodic orbits
$\Gamma_{h}:H=h$, $h\in(2-b^2,2)$, which surrounds the center $C$. $H=2-b^2$
is the value of the center.

In what follows we are going to study the polynomial perturbations of these two subclasses of system \eqref{eq11}:
\begin{equation}\begin{array}{ll}\label{eq12}
\dot{x}=x(1+bx+x^2-y^2)+\varepsilon f(x,y),\\[2ex]
\dot{y}=y(-1-cy+x^2-y^2)+\varepsilon g(x,y),
\end{array}\end{equation}
where $f(x,y)$ and $g(x,y)$ are polynomials in the variables $x$ and
$y$ with $\mbox{max}\{\deg f$, $\deg g\}=n$.

Consider \eqref{eq12} with $0<\varepsilon\ll1$. Let $H_{X_{2.9}}(n)$ and $H_{X_{2.10}}(n)$ be the maximum number of zeros (taking into account their
multiplicity) of the Abelian integrals associated with the two systems:
\[
\begin{array}{ll} \vspace{.1in}
X_{2.9}, & \hbox{i.e.,  $0\leq b<c<2$,}\\
X_{2.10},&  \hbox{i.e., $0<b=c<2$, }
\end{array}
\]
respectively. Note that system $X_{2.9}$ has two period annuli and here $H_{X_{2.9}}(n)$ denotes the maximum number of zeros of Abelian integral $M(h)$ on
the interval $(-2,h-)$ or $(h_+,2)$.

The main results of this paper are:

\begin{theorem}\label{th1}
$H_{X_{2.9}}(n)\leq1$ if $n\leq 2$ and $H_{X_{2.9}}(n)\leq2n-3$ if $n\geq 3$.
\end{theorem}

\begin{theorem}\label{th2}
$H_{X_{2.10}}(n)\leq1$ if $n\leq 2$ and $H_{X_{2.10}}(n)\leq[(3n-3)/2]$ if $n\geq 3$.
\end{theorem}

\begin{remark}
By Proposition 1 of \cite{LLLZ2}, it is easy to
verify that the systems $X_{2.9}$ and $X_{2.10}$ can be transformed into $S^{*}$:
\begin{equation*}\begin{array}{ll}
\dot{x}=-y+\beta x^2-2\alpha xy-\beta y^2+x^2y,\\[2ex]
\dot{y}=x+\alpha x^2+2\beta xy-\alpha y^2+xy^2,
\end{array}\end{equation*}
in Table I of \cite{LLLZ2}. Thus the centers of these two systems are isochronous.
Moreover, it follows from the first integral that system $X_{2.10}$ is
reversible (with respect to the straight line $y=x$) while system $X_{2.9}$ is not. In fact, by the affine
transformation
\[(x,y,t)\rightarrow\left(y-x,\dfrac{b(x+y)+2}{\sqrt{4-b^2}},-\dfrac{\sqrt{4-b^2}}{b}t\right),
\]
system $X_{2.10}$ is transformed into $S^{*}$ with
$\alpha=0$, $\beta=-\sqrt{4-b^2}/2$, which belongs to the reversible case (A) of \cite{LLLZ2}.
As a special one, our result
is better than that of the general case (see Theorem 2 of \cite{LLLZ2}).
Similarly, by the affine
transformation
\[(x,y,t)\rightarrow\left(\dfrac{x+y-\xi-\eta}{\sqrt{1-(\xi+\eta)^2}},\dfrac{y-x+\xi-\eta}{\sqrt{(\xi-\eta)^2-1}},-\sqrt{4\xi^2\eta^2-(\xi^2+\eta^2-1)^2}t\right),
\]
system $X_{2.9}$ is reduced to $S^{*}$ with
$\alpha=(\xi+\eta)\sqrt{1-(\xi+\eta)^2}/(4\xi\eta)$,
$\beta=(\xi-\eta)\sqrt{(\xi-\eta)^2-1}/(4\xi\eta)$, where
$\xi=-\delta\zeta$, $\eta=\gamma\zeta$,
$\zeta=(b\delta+c\gamma)/2(\delta^2-\gamma^2)$. Since $\alpha\beta\neq
0$, system $X_{2.9}$ is not reversible. The problem that how many limit cycles bifurcate from the period annuli
of system $X_{2.9}$ has not yet been studied.
\end{remark}
Different from the classical methods which are used to study the number of
limit cycles that bifurcate from the periodic orbits of a center,
for example Poincar\'e-Melnikov integral method, Picard-Fuchs equation method \cite{HI,ZZ}, inverse integrating
factor method \cite{GLV}, averaging method \cite{BL1} and so on, this paper takes
advantage of some symmetric properties of the first integral and
integrating factor to give the number of zeros of the Abelian integrals by direct
computation.  Compared with the method using Picard-Fuchs
equation, our method gives the better upper bound. The computational approach has similarities with the technique used in \cite{LLLZ1,LLLZ2},
but we do not use Green's theorem which makes the work
harder than computing the usual single Abelian integral directly in this paper.

Recently, Llibre {\it et al.} also study limit cycles of cubic polynomial differential systems with rational first integrals of degree 2 under polynomial perturbations of degree $3$ in \cite{LLM} using averaging method. They give six families of cubic polynomial differential systems, denote by $P_k$ for $k=1,2,\ldots,6$, where $P_3$ and $P_5$ are equivalent to ours. However, they only give the exact upper bound for $P_1,P_2,P_4$ and $P_6$, and give an example of class $P_5$ that has at most 3 limit cycles.

We investigate the sharp upper bound for the number of zeros of the Abelian integrals
with respect to the two systems for $n=3$. The result is the following one.

\begin{theorem}\label{th3}
$H_{X_{2.9}}(3)=3$ and $H_{X_{2.10}}(3)=3$.
\end{theorem}
As is introduced above, there are two families of periodic orbits of system $X_{2.9}$.
It is natural to consider the simultaneous bifurcation and distribution of
limit cycles that emerge from both period annuli (e.g., \cite{CGP,GGJ}).
The configuration of limit cycles $(u,v)$, $u\geq0, v\geq0$, is considered to be achievable
if, for $\varepsilon$ small enough, exactly $u$ (resp. $v$) limit cycles bifurcate from the periodic orbits surrounding $C_{+}$ (resp. $C_{-}$).
Though, up to first order in $\varepsilon$, three limit cycles can emerge from each period annulus of system $X_{2.9}$ under cubic polynomial perturbations, it does not means that the configuration (3,3) can be achievable (see for instance \cite{CLP,L}).  In the present paper, we give a positive answer, i.e., (3,3) is impossible.
\begin{theorem}\label{th4}
Under cubic polynomial perturbations, all configurations $(u,v)$ of limit cycles bifurcated from the two period annuli of system
$X_{2.9}$ can be realized, where $0\leq u, v\leq 3, u+v\leq5$.
\end{theorem}
The paper is organized as follows: in sections 2 and 3, we study the number of zeros
of the Abelian integrals for systems $X_{2.9}$ and $X_{2.10}$, respectively. Upper bounds for $H_{X_{2.9}}(n)$ and $H_{X_{2.10}}(n)$ are obtained. Section 4 focuses on the analysis of the least upper bound for the number of zeros of the Abelian integrals associated to the two systems for $n=3$. \emph{Chebyshev criterion} is used to determine $H_{X_{2.9}}(3)$ and $H_{X_{2.10}}(3)$. Finally, we give the simultaneous bifurcation and distribution of limit cycles bifurcated from the two period annuli of system
$X_{2.9}$ under cubic polynomial perturbations in section 5.

\section{Zeros of the Abelian integral for system $X_{2.9}$}
In this section, we will study $H_{X_{2.9}}(n)$.
Let $\Gamma_{h}$ be the closed component of the algebraic
curve
\[
H(x,y,b,c)=\dfrac{1+bx+cy+x^2+y^2}{xy}=h
\]
and $\Gamma_{h}^{'}$ be the closed component
of the algebraic curve
\[
H(x,y,c,b)=\dfrac{1+cx+by+x^2+y^2}{xy}=h.
\]
The Abelian integral associated to system $X_{2.9}$ is defined as
\begin{equation}\begin{split}\label{Mh:29}
M(h)&=\displaystyle\oint_{\Gamma_{h}}x^{-2}y^{-2}(f\mathrm{d}y-g\mathrm{d}x)\\
&=\displaystyle\sum_{i+j=0}^{n}\oint_{\Gamma_{h}}(a_{ij}x^{i-2}y^{j-2}\mathrm{d}y-b_{ij}x^{i-2}y^{j-2}\mathrm{d}x), \quad  h\in(-2,h_{-})\cup(h_{+},2),
\end{split}\end{equation}
where $\Gamma_{h}$ has the positive (resp. negative) orientation if it surrounds $C_{+}$ (resp. $C_{-}$).

Denote
\begin{equation}\label{Iij:29}
I_{ij}(h)=\displaystyle\oint_{\Gamma_{h}}x^{i-2}y^{j-2}\mathrm{d}x,\quad \bar{I}_{ij}(h)=\displaystyle\oint_{\Gamma_{h}^{'}}x^{i-2}y^{j-2}\mathrm{d}x.
\end{equation}
\begin{proposition}\label{Prop:29}
$M(h)$ has the following expression:
\begin{equation}
M(h)=\sum_{i+j=0}^{n}\left(-a_{ji}\bar{I}_{ij}(h)-b_{ij}I_{ij}(h)\right), \quad  h\in(-2,h_{-})\cup(h_{+},2).
\end{equation}
\end{proposition}
\begin{proof}
Suppose the level curve $\Gamma_{h}$ is a periodic orbit that
surrounds the center $C_{+}$ of system $X_{2.9}$, then
\begin{equation*}\begin{split}
&\displaystyle\oint_{\Gamma_{h}}x^{i-2}y^{j-2}\mathrm{d}y\\
=&\displaystyle\int_{y_{1}}^{y_{2}}\left(\dfrac{hy-b}{2}+\sqrt{\Psi(y)
}\right)^{i-2}y^{j-2}\mathrm{d}y+\displaystyle\int_{y_{2}}^{y_{1}}\left(\dfrac{hy-b}{2}-\sqrt{\Psi(y)
}\right)^{i-2}y^{j-2}\mathrm{d}y,
\end{split}\end{equation*}
where $\Psi(y)=(h^{2}/4-1)y^{2}-(bh/2+c)y+b^{2}/4-1$ and $y_{1},
y_{2}$ denote the two roots of the equation $\Psi(y)=0$ with $y_{1}<y_{2}$.
Set $y=x$ in the right-hand-side of the last equality, then
\begin{equation}\begin{split}\label{1}
&\displaystyle\oint_{\Gamma_{h}}x^{i-2}y^{j-2}\mathrm{d}y\\[-3pt]
=&\displaystyle\int_{y_{1}}^{y_{2}}\left(\dfrac{hx-b}{2}+\sqrt{\Psi(x)
}\right)^{i-2}x^{j-2}\mathrm{d}x+\displaystyle\int_{y_{2}}^{y_{1}}\left(\dfrac{hx-b}{2}-\sqrt{\Psi(x)
}\right)^{i-2}x^{j-2}\mathrm{d}x\\[-3pt]
=&-\displaystyle\oint_{\Gamma_{h}^{'}}y^{i-2}x^{j-2}\mathrm{d}x=-\bar{I}_{ji}(h),
\end{split}\end{equation}
where $\Gamma_{h}^{'}$ has the same orientation as $\Gamma_{h}$. The second
equality holds since $y_{1}, y_{2}$ are also the roots of the equation $\Psi(x)=0$.

The proposition follows from \eqref{Mh:29}, \eqref{Iij:29} and \eqref{1}.
\end{proof}
By Proposition 6, owing to the fact that the form and algorithm of the two parts
of $M(h)$ are similar, we only need to calculate
\begin{equation}\label{M29}
\sum_{i+j=0}^{n}b_{ij}I_{ij}(h)=M_{1}(h)+M_{2}(h)+M_{3}(h), \quad h\in(-2,h_{-})\cup(h_{+},2),
\end{equation}
where
\begin{equation*}\begin{split}
M_{1}(h)&=\displaystyle\sum_{j\geq2,\atop i+j\leq n}b_{ij}I_{ij}(h)=\sum_{j\geq2,\atop i+j\leq n}b_{ij}\displaystyle\oint_{\Gamma_{h}}x^{i-2}y^{j-2}\mathrm{d}x,\\[-3pt]
M_{2}(h)&=\displaystyle\sum^{n-1}_{i=0}b_{i1}I_{i1}(h)=\sum_{i=0}^{n-1}b_{i1}\displaystyle\oint_{\Gamma_{h}}x^{i-2}y^{-1}\mathrm{d}x,\\[-3pt]
M_{3}(h)&=\displaystyle\sum^{n}_{i=0}b_{i0}I_{i0}(h)=\sum_{i=0}^{n}b_{i0}\displaystyle\oint_{\Gamma_{h}}x^{i-2}y^{-2}\mathrm{d}x.
\end{split}\end{equation*}

Let $\Delta=(h^{2}/4-1)x^{2}-(ch/2+b)x+c^{2}/4-1$ and $x_{1}, x_{2}$
are the two roots of the equation $\Delta=0$ with $x_{1}<
x_{2}$. Then
\begin{equation}\label{Viete1}
x_{1}+x_{2}=-\dfrac{2(2b+ch)}{4-h^{2}},\ \ \
x_{1}x_{2}=\dfrac{4-c^{2}}{4-h^{2}}.
\end{equation}
We have that
\begin{equation}\label{sign}
I_{ij}(h)=\displaystyle\pm\int_{x_{1}}^{x_{2}}x^{i-2}\left[\left(\dfrac{hx-c}{2}-\sqrt{\Delta
}\right)^{j-2}-\left(\dfrac{hx-c}{2}+\sqrt{\Delta }\right)^{j-2}\right]\mathrm{d}x.
\end{equation}
\begin{remark}
The appearance of ``$\pm$'' in \eqref{sign} comes from the fact that $\Gamma_{h}$  has the positive orientation when $h\in(h_+,2)$ and the negative orientation when $h\in(-2,h_-)$. However, it is easy to verify that the sign ``$\pm$'' does not affect the maximum number of zeros of the Abelian integral defined on each interval. Thus we will drop it in the computation of $M_{i}(h),i=1,2,3$ for $h\in(-2,h_{-})\cup(h_{+},2)$. The results obtained differ from the original functions at most by a negative sign.
\end{remark}
First, we get that
\begin{equation}\begin{split}\label{M1}
M_{1}(h)&=\displaystyle\sum_{j\geq2, \atop i+j\leq n} b_{ij}\int_{x_{1}}^{x_{2}}x^{i-2}\left[\left(\dfrac{hx-c}{2}-\sqrt{\Delta
}\right)^{j-2}-\left(\dfrac{hx-c}{2}+\sqrt{\Delta }\right)^{j-2}\right]\mathrm{d}x\\[-10pt]
&=\displaystyle\sum_{j\geq3, \atop i+j\leq n} b_{ij}\int_{x_{1}}^{x_{2}}x^{i-2}\left(\sum_{k=0}^{j-3}\tilde{m}_{j,k}(h)x^{k}\sqrt{\Delta}\right)\mathrm{d}x\\[-4pt]
&=\displaystyle\sum_{k=0}^{n-3}m_{k}(h)J_{k}(h),
\end{split}\end{equation}
where $\tilde{m}_{j,k}(h),m_{k}(h)$ are polynomials of $h$ with
$\deg \tilde{m}_{j,k}(h)\leq{k},\deg m_{k}(h)\leq{k}$, and
\begin{equation*}
J_{k}(h)=\displaystyle\int_{x_{1}}^{x_{2}}x^{k-2}\sqrt{\Delta}\mathrm{d}x
=\sqrt{1-h^{2}/4}\int_{x_{1}}^{x_{2}}x^{k-2}\sqrt{(x_{2}-x)(x-x_{1})}\mathrm{d}x.
\end{equation*}
In order to study $J_{k}(h)$, we use two different transformations.
If $\sqrt{(x_{2}-x)(x-x_{1})}=t(x-x_{1})$, then
\begin{equation*}
J_{k}(h)=\displaystyle\sqrt{4-h^{2}}\
(x_{2}-x_{1})^{2}\int_{0}^{\infty}\dfrac{t^{2}(x_{2}+t^{2}x_{1})^{k-2}}{(1+t^{2})^{k+1}}\mathrm{d}t,
\end{equation*}
and if $\sqrt{(x_{2}-x)(x-x_{1})}=t(x_2-x)$, then
\begin{equation*}
J_{k}(h)=\displaystyle\sqrt{4-h^{2}}\ (x_{2}-x_{1})^{2}\int_{0}^{\infty}\dfrac{t^{2}(x_{1}+t^{2}x_{2})^{k-2}}{(1+t^{2})^{k+1}}\mathrm{d}t.
\end{equation*}
Thus, by \eqref{Viete1}, for $k\geq2$,
\begin{equation}\begin{split}\label{Jk}
J_{k}(h)&=\displaystyle\sqrt{4-h^{2}}\ (x_{2}-x_{1})^{2}\int_{0}^{\infty}\dfrac{t^{2}[(x_{2}+t^{2}x_{1})^{k-2}+(x_{1}+t^{2}x_{2})^{k-2}]}{2(1+t^{2})^{k+1}}\mathrm{d}t\\
&=\displaystyle\sqrt{4-h^{2}}\
(x_{2}-x_{1})^{2}\sum_{i+2j=k-2}d_{ij}(x_{1}+x_{2})^{i}(x_{1}x_{2})^{j}\\
&=\displaystyle\dfrac{16(h^{2}+bch+b^{2}+c^{2}-4)}{(4-h^{2})^{3/2}}
\sum_{i+2j=k-2}d_{ij}\left[-\dfrac{2(2b+ch)}{4-h^{2}}\right]^{i}\left(\dfrac{4-c^{2}}{4-h^{2}}\right)^{j}\\
&=(4-h^{2})^{-(k-1/2)}P_{k}(h),
\end{split}\end{equation}
where $d_{ij}, i+2j=k-2$ are constants and $P_{k}(h)$ denotes a polynomial of degree $k$. By direct computations, we get that
\begin{equation}\begin{array}{ll}\label{J1J0}
J_{1}(h)=\displaystyle\int_{x_{1}}^{x_{2}}x^{-1}\sqrt{\Delta}\mathrm{d}x
=\left\{\begin{array}{ll}
-\displaystyle\dfrac{2b+ch}{2\sqrt{4-h^{2}}}\pi-\dfrac{\sqrt{4-c^{2}}}{2}\pi, & \mbox{$h\in(-2,h_{-})$,}\\[2ex]
-\displaystyle\dfrac{2b+ch}{2\sqrt{4-h^{2}}}\pi+\dfrac{\sqrt{4-c^{2}}}{2}\pi, & \mbox{$h\in(h_{+},2)$,}\\[2ex]
\end{array} \right.\\[2ex]
J_{0}(h)=\displaystyle\int_{x_{1}}^{x_{2}}x^{-2}\sqrt{\Delta}\mathrm{d}x
=\left\{\begin{array}{ll}
-\displaystyle\dfrac{2b+ch}{2\sqrt{4-c^{2}}}\pi-\dfrac{\sqrt{4-h^{2}}}{2}\pi, & \mbox{$h\in(-2,h_{-})$,}\\[2ex]
\displaystyle\dfrac{2b+ch}{2\sqrt{4-c^{2}}}\pi-\dfrac{\sqrt{4-h^{2}}}{2}\pi, & \mbox{$h\in(h_{+},2)$.}\\[2ex]
\end{array} \right.
\end{array}\end{equation}
It follows from \eqref{M1}, \eqref{Jk} and \eqref{J1J0} that
\begin{equation}\begin{split}\label{M1_9}
M_{1}(h)&=m_{0}(h)J_{0}(h)+m_{1}(h)J_{1}(h)+\displaystyle\sum_{k=2}^{n-3}m_{k}(h)J_{k}(h)\\[-6pt]
&=\hat{P}_{1}(h)+\displaystyle\dfrac{\hat{P}_{2n-6}(h)}{(4-h^{2})^{n-7/2}},
\end{split}\end{equation}
where $\hat{P}_{1}(h)$ and $\hat{P}_{2n-6}(h)$ are
polynomials with
$\deg \hat{P}_{1}(h)\leq{1}$ and $\deg\hat{P}_{2n-6}(h)\leq2n-6$, respectively.

Next we calculate $M_{2}(h)$ and $M_{3}(h)$.
\begin{equation}\begin{split}\label{M2}
M_{2}(h)&=\displaystyle\sum^{n-1}_{i=0} b_{i1}\int_{x_{1}}^{x_{2}}x^{i-2}\left[\dfrac{1}{(hx-c)/2-\sqrt{\Delta
}}-\dfrac{1}{(hx-c)/2+\sqrt{\Delta}}\right]\mathrm{d}x\\[-3pt]
&=\displaystyle \sum^{n-1}_{i=0} b_{i1}\int_{x_{1}}^{x_{2}}\dfrac{2x^{i-2}\sqrt{\Delta
}}{x^{2}+bx+1}\mathrm{d}x\\[-3pt]
&=\displaystyle \sum^{3}_{i=0}\tilde{b}_{i1}S_{i}(h)+\sum^{n-1}_{i=4}\tilde{b}_{i1}\int_{x_{1}}^{x_{2}}x^{i-4}\sqrt{\Delta
}\mathrm{d}x\\[-3pt]
&=\displaystyle\sum^{3}_{i=0}\tilde{b}_{i1}S_{i}(h)+ \sum^{n-1}_{i=4}\tilde{b}_{i1}J_{i-2}(h),\\[-3pt]
\end{split}\end{equation}
where
$\tilde{b}_{i1},i=0,1,...,n-1,$ are linear combinations of $b_{i1},i=0,1,...,n-1,$ and

\begin{equation*}
S_{i}(h)=\displaystyle\int_{x_{1}}^{x_{2}}\dfrac{2x^{i-2}\sqrt{\Delta
}}{x^{2}+bx+1}\mathrm{d}x, \quad i=0,1,2,3.
\end{equation*}
If $h\in(-2,h_{-})$, then
\begin{equation}\begin{array}{ll}\label{Si-}
S_{0}(h)=\dfrac{(b^{2}-2)c+bh}{\sqrt{4-b^{2}}}\pi-\dfrac{(c^{2}-2)b+ch}{\sqrt{4-c^{2}}}\pi,\\
S_{1}(h)=-\dfrac{bc+2h}{\sqrt{4-b^{2}}}\pi-\sqrt{4-c^{2}}\pi,\\
S_{2}(h)=\dfrac{2c+bh}{\sqrt{4-b^{2}}}\pi-\sqrt{4-h^{2}}\pi,\\
S_{3}(h)=-\dfrac{bc+(b^{2}-2)h}{\sqrt{4-b^{2}}}\pi-\dfrac{-2b+ch+bh^{2}}{\sqrt{4-h^{2}}}\pi,
\end{array}\end{equation}
and if $h\in(h_{+},2)$, then
\begin{equation}\begin{array}{ll}\label{Si+}
S_{0}(h)=\dfrac{(b^{2}-2)c+bh}{\sqrt{4-b^{2}}}\pi+\dfrac{(c^{2}-2)b+ch}{\sqrt{4-c^{2}}}\pi,\\
S_{1}(h)=-\dfrac{bc+2h}{\sqrt{4-b^{2}}}\pi+\sqrt{4-c^{2}}\pi,
\end{array}\end{equation}
$S_{2}(h)$ and $S_{3}(h)$ are the same as \eqref{Si-}.
By \eqref{Jk}, \eqref{M2}, \eqref{Si-} and \eqref{Si+}, we obtain
\begin{equation}\begin{split}\label{M2_9}
M_{2}(h)
=\tilde{P}_{1}(h)+\displaystyle\dfrac{\tilde{P}_{2n-6}(h)}{(4-h^{2})^{n-7/2}},\\
\end{split}\end{equation}
where $\tilde{P}_{1}(h)$ and $\tilde{P}_{2n-6}(h)$ are
polynomials with
$\deg \tilde{P}_{1}(h)\leq{1}$ and $\deg\tilde{P}_{2n-6}(h)\leq2n-6$, respectively.

The technique used in the computation of $M_{3}(h)$ is shown as follows:
\begin{equation}\begin{split}\label{M3}
M_{3}(h)&=\displaystyle\sum^{n}_{i=0}b_{i0}\int_{x_{1}}^{x_{2}}x^{i-2}\left[\dfrac{1}{((hx-c)/2-\sqrt{\Delta
})^2}-\dfrac{1}{((hx-c)/2+\sqrt{\Delta})^2}\right]\mathrm{d}x\\[-3pt]
&=\displaystyle\sum^{n}_{i=0}b_{i0}\int_{x_{1}}^{x_{2}}\dfrac{2x^{i-2}(hx-c)\sqrt{\Delta
}}{(x^{2}+bx+1)^2}\mathrm{d}x\\[-3pt]
&=\displaystyle\sum^{3}_{i=0}b_{i0}R_{i}(h)+\sum^{n}_{i=4}b_{i0}\int_{x_{1}}^{x_{2}}\dfrac{2x^{i-2}(hx-c)\sqrt{\Delta
}}{(x^{2}+bx+1)^2}\mathrm{d}x\\[-3pt]
&=\mu(h)S_{2}(h)+\nu(h)S_{3}(h)+\displaystyle\sum^{3}_{i=0}\bar{b}_{i0}R_{i}(h)+\sum^{n}_{i=5}\bar{b}_{i0}\omega_{i}(h)J_{i-3}(h),
\end{split}\end{equation}
where
$\bar{b}_{i0}, i=0,1,2,3,5,...,n,$ are
linear combinations of $b_{i0}, i=0,1,...,n$. $\mu(h)$, $\nu(h)$ and
$\omega_{i}(h)$ are polynomials of degree at most $1$, and
\begin{equation*}
R_{i}(h)=\displaystyle\int_{x_{1}}^{x_{2}}\dfrac{2x^{i-2}(hx-c)\sqrt{\Delta
}}{(x^{2}+bx+1)^2}\mathrm{d}x, \quad i=0,1,2,3.
\end{equation*}
Further calculations show that
when $h\in(-2,h_{-})$,
\begin{equation}\begin{split}\label{Ri-}
R_{0}(h)=&\dfrac{2(-8+6c^2-6b^2(-1+c^2)+b^4(-1+c^2))\pi}{(4-b^2)^{3/2}}+\dfrac{2bc(-3+c^2)\pi}{\sqrt{4-c^2}}\\
&+h\left(\dfrac{2b(-6+b^2)c\pi}{(4-b^2)^{3/2}}+\dfrac{2(-2+c^2)\pi}{\sqrt{4-c^2}}\right)-\dfrac{4h^2\pi}{(4-b^2)^{3/2}},\\
R_{1}(h)=&\dfrac{2bc^2\pi}{(4-b^2)^{3/2}}+\dfrac{b(c^2-2)\pi}{\sqrt{4-b^2}}+c\sqrt{4-c^2}\pi+\dfrac{8ch\pi}{(4-b^2)^{3/2}}+\dfrac{2bh^2\pi}{(4-b^2)^{3/2}},\\
R_{2}(h)=&\dfrac{-4(-4+b^2+c^2+bch+h^2)\pi}{(4-b^2)^{3/2}},\\
R_{3}(h)=&\dfrac{2b(-4+b^2+c^2)\pi}{(4-b^2)^{3/2}}+\dfrac{8ch\pi}{(4-b^2)^{3/2}}
-\dfrac{b(-6+b^2)h^2\pi}{(4-b^2)^{3/2}}-h\sqrt{4-h^2}\pi,
\end{split}\end{equation}
and when $h\in(h_{+},2)$,
\begin{equation}\begin{split}\label{Ri+}
R_{0}(h)=&\dfrac{2(-8+6c^2-6b^2(-1+c^2)+b^4(-1+c^2))\pi}{(4-b^2)^{3/2}}-\dfrac{2bc(-3+c^2)\pi}{\sqrt{4-c^2}}\\
&+h\left(\dfrac{2b(-6+b^2)c\pi}{(4-b^2)^{3/2}}-\dfrac{2(-2+c^2)\pi}{\sqrt{4-c^2}}\right)-\dfrac{4h^2\pi}{(4-b^2)^{3/2}},\\
R_{1}(h)=&\dfrac{2bc^2\pi}{(4-b^2)^{3/2}}+\dfrac{b(c^2-2)\pi}{\sqrt{4-b^2}}-c\sqrt{4-c^2}\pi+\dfrac{8ch\pi}{(4-b^2)^{3/2}}+\dfrac{2bh^2\pi}{(4-b^2)^{3/2}},
\end{split}\end{equation}
$R_{2}(h)$ and $R_{3}(h)$ are the same as \eqref{Ri-}.
Hence by \eqref{Jk}, \eqref{Si-}, \eqref{Si+}, \eqref{M3}, \eqref{Ri-} and \eqref{Ri+}, we get
\begin{equation}\begin{split}\label{M3_9}
M_{3}(h)=\bar{P}_{2}(h)+\displaystyle\dfrac{{\bar{P}_{2n-5}(h)}}{(4-h^{2})^{n-7/2}},
\end{split}\end{equation}
where $\bar{P}_{2}(h)$ and $\bar{P}_{2n-5}(h)$ stand for polynomials of $h$ with
$\deg \bar{P}_{2}(h)\leq{2}$ and $\deg\bar{P}_{2n-5}(h)\leq2n-5$, respectively.

\noindent{\bfseries{Proof of Theorem 1.}} From \eqref{M29}, \eqref{M1_9}, \eqref{M2_9},
\eqref{M3_9} and Proposition 6, together with the similar form and algorithm of the two parts of
$M(h)$, we get the Abelian integral
\begin{equation}M(h)=\label{M}\left\{\begin{array}{ll}
Q_{2}(h)+\displaystyle\dfrac{{Q_{2n-5}(h)}}{(4-h^{2})^{n-7/2}} & \mbox{if $n\geq3$,}\\[2ex]
\widetilde{Q}_{2}(h) & \mbox{if $n\leq2$,}\\[2ex]
\end{array} \right.
\end{equation}
where $Q_{2}(h)$, $\widetilde{Q}_{2}(h)$ and $Q_{2n-5}(h)$ are polynomials with $\deg Q_{2}(h)\leq{2}$, $\deg \widetilde{Q}_{2}(h)\leq{2}$
and $\deg Q_{2n-5}(h)\leq2n-5$, respectively. We remark that $M(h)$ has the same form as \eqref{M} on $(-2,h_-)$ and $(h_+,2)$, respectively. But on these two intervals, they are different functions (it is clear to see from \eqref{J1J0}, \eqref{Si-}, \eqref{Si+} and so on).

Obviously, $M(h)$ is  a polynomial of degree at most 2 if $n\leq2$ and it
has at most 2 zeros. Moreover, it vanishes at the value of the
center. Thus it has at most one zero in each open interval.

$M(h)$ does not have a rational form if $n\geq3$. But we can get
the estimation of the number of zeros of $M(h)$ by taking derivatives three
times. Therefore
\begin{equation*}
\#\{h|M^{(3)}(h)=0\}\leq 2n-5,
\end{equation*}
where $\#$ denotes the number of elements of a finite set.
On the other hand, note that $M(h_{\pm})=0$.
Hence
\begin{equation*}\begin{split}
&\mbox{max}\{\#\{h\in (-2,h_-)|M(h)=0\},\#\{h\in (h_+,2)|M(h)=0\}\}\\
\leq & 2n-5+3-1=2n-3,
\end{split}\end{equation*}
i.e.,
\begin{equation*}
H_{X_{2.9}}(n)\leq 2n-3.
\end{equation*}

\section{Zeros of the Abelian integral for system $X_{2.10}$}
This section is devoted to investigate the polynomial perturbations of system $X_{2.10}$, i.e., system (1.3) with $b=c$.

Firstly, we have the following result by the same argument as the proof of Proposition 6.
\begin{proposition}
The Abelian integral associated to the system
$X_{2.10}$ is
\begin{equation}\begin{split}
M(h)&=\displaystyle\oint_{\Gamma_{h}}x^{-2}y^{-2}(f\mathrm{d}y-g\mathrm{d}x)\\[-3pt]
&=\displaystyle\sum_{i+j=0}^{n}\oint_{\Gamma_{h}}(a_{ij}x^{i-2}y^{j-2}\mathrm{d}y-b_{ij}x^{i-2}y^{j-2}\mathrm{d}x)\\[-3pt]
&=\displaystyle\sum_{i+j=0}^{n}(-a_{ji}-b_{ij})I_{ij}(h), \quad h\in(2-b^2,2),
\end{split}\end{equation}
where $\Gamma_{h}$ is the closed component of the algebraic curve $1+bx+by+x^2+y^2=hxy$ with the positive orientation, and
\[
I_{ij}(h)=\displaystyle\oint_{\Gamma_{h}}x^{i-2}y^{j-2}\mathrm{d}x.
\].
\end{proposition}
Because the algorithm of  $M(h)$ is the same as that in Section 2, we only give the main results.
\begin{equation}
M(h)=\sum_{i+j=0}^{n}c_{ij}I_{ij}(h)=M_{1}(h)+M_{2}(h)+M_{3}(h), \quad h\in(2-b^2,2),
\end{equation}
where $c_{ij}=-a_{ji}-b_{ij}$, and
\begin{equation*}\begin{split}
M_{1}(h)&=\displaystyle\sum_{j\geq2,\atop i+j\leq n}c_{ij}I_{ij}(h)=\sum_{j\geq2,\atop i+j\leq n}c_{ij}\displaystyle\oint_{\Gamma_{h}}x^{i-2}y^{j-2}\mathrm{d}x,\\[-3pt]
M_{2}(h)&=\displaystyle\sum^{n-1}_{i=0}c_{i1}I_{i1}(h)=\sum_{i=0}^{n-1}c_{i1}\displaystyle\oint_{\Gamma_{h}}x^{i-2}y^{-1}\mathrm{d}x,\\[-3pt]
M_{3}(h)&=\displaystyle\sum^{n}_{i=0}c_{i0}I_{i0}(h)=\sum_{i=0}^{n}c_{i0}\displaystyle\oint_{\Gamma_{h}}x^{i-2}y^{-2}\mathrm{d}x.
\end{split}\end{equation*}

Denote $\Delta=(h^{2}/4-1)x^{2}-(bh/2+b)x+b^{2}/4-1$ and let $x_{1}, x_{2}$
be the two roots of the equation $\Delta=0$ with $x_{1}<
x_{2}$. Then
\begin{equation}
x_{1}+x_{2}=-\dfrac{2b}{2-h},\ \ \
x_{1}x_{2}=\dfrac{4-b^{2}}{4-h^{2}}.
\end{equation}
Consequently,
\begin{equation}
M_{1}(h)=\displaystyle\sum_{k=0}^{n-3}m_{k}(h)J_{k}(h),
\end{equation}
where $m_{k}(h)$ is a polynomial of $h$ with
$\deg m_{k}(h)\leq{k}$ and
\begin{equation*}
J_{k}(h)=\displaystyle\int_{x_{1}}^{x_{2}}x^{k-2}\sqrt{\Delta}\mathrm{d}x.
\end{equation*}
It is easy to verify that if $k\geq2$,
\begin{equation}\begin{split}
J_{k}(h)&=\displaystyle\sqrt{4-h^{2}}\
(x_{2}-x_{1})^{2}\sum_{i+2j=k-2}\bar{d}_{ij}(x_{1}+x_{2})^{i}(x_{1}x_{2})^{j}\\
&=\displaystyle\dfrac{16(h-(2-b^2))}{(2-h)\sqrt{4-h^{2}}}
\sum_{i+2j=k-2}\bar{d}_{ij}\left(-\dfrac{2b}{2-h}\right)^{i}\left(\dfrac{4-b^{2}}{4-h^{2}}\right)^{j}\\
&=\displaystyle\dfrac{16(h-(2-b^2))}{\sqrt{4-h^{2}}(2-h)^{k-1}(2+h)^{[(k-2)/2]}}P_{[(k-2)/2]}(h),
\end{split}\end{equation}
where $\bar{d}_{ij}, i+2j=k-2$ are constants, $[(k-2)/2]$ denotes the integer part of $(k-2)/2$ and $P_{[(k-2)/2]}(h)$ is a polynomial of degree $[(k-2)/2]$.
A direct computation leads to
\begin{equation}\begin{array}{ll}
J_{1}(h)=-\displaystyle\dfrac{b(2+h)}{2\sqrt{4-h^{2}}}\pi+\dfrac{\sqrt{4-b^{2}}}{2}\pi,\\[2ex]
J_{0}(h)=\displaystyle\dfrac{b(2+h)}{2\sqrt{4-b^{2}}}\pi-\dfrac{\sqrt{4-h^{2}}}{2}\pi.
\end{array}\end{equation}
It follows from (3.4), (3.5) and (3.6) that
\begin{equation}
M_{1}(h)=\hat{P}_{1}(h)+\displaystyle\dfrac{\hat{P}_{[(3n-9)/2]}(h)}{\sqrt{4-h^{2}}(2-h)^{n-4}(2+h)^{[(n-5)/2]}},
\end{equation}
where $\hat{P}_{1}(h)$ and $\hat{P}_{[(3n-9)/2]}(h)$ are
polynomials with $\deg \hat{P}_{1}\leq{1}$ and $\deg\hat{P}_{[(3n-9)/2]}\leq[(3n-9)/2]$, respectively.

In addition,
\begin{equation}\begin{split}
M_{2}(h)&=\displaystyle\sum^{n-1}_{i=0}c_{i1}\int_{x_{1}}^{x_{2}}x^{i-2}\left(\dfrac{1}{(hx-b)/2-\sqrt{\Delta
}}-\dfrac{1}{(hx-b)/2+\sqrt{\Delta}}\right)\mathrm{d}x\\
&=\displaystyle\sum^{3}_{i=0}\tilde{c}_{i1}S_{i}(h)+\sum^{n-1}_{i=4}\tilde{c}_{i1}J_{i-2}(h),
\end{split}\end{equation}
where
$\tilde{c}_{i1}$, $i=0,1,...,n-1$ are
linear combinations of $c_{i1}$, $i=0,1,...,n-1$, and
\begin{equation*}
S_{i}(h)=\displaystyle\int_{x_{1}}^{x_{2}}\dfrac{2x^{i-2}\sqrt{\Delta
}}{x^{2}+bx+1}\mathrm{d}x, \quad i=0,1,2,3.
\end{equation*}
By direct calculation we obtain that
\begin{equation}\begin{array}{ll}
S_{0}(h)=-bS_{1}(h)=\dfrac{2b(h-(2-b^2))}{\sqrt{4-b^{2}}}\pi,\\
S_{2}(h)=\dfrac{b(2+h)}{\sqrt{4-b^{2}}}\pi-\sqrt{4-h^{2}}\pi,\\
S_{3}(h)=-\dfrac{b^2+(b^{2}-2)h}{\sqrt{4-b^{2}}}\pi-\dfrac{b(-2+h+h^{2})}{\sqrt{4-h^{2}}}\pi.
\end{array}\end{equation}
Therefore, we get from (3.5), (3.8) and (3.9) that
\begin{equation}\begin{array}{ll}
M_{2}(h)=\tilde{P}_{1}(h)+\displaystyle\dfrac{\tilde{P}_{[(3n-9)/2]}(h)}{\sqrt{4-h^{2}}(2-h)^{n-4}(2+h)^{[(n-5)/2]}},
\end{array}\end{equation}
where $\tilde{P}_{1}(h)$ and $\tilde{P}_{[(3n-9)/2]}(h)$ are
polynomials with $\deg \tilde{P}_{1}\leq{1}$ and $\deg\tilde{P}_{[(3n-9)/2]}\leq[(3n-9)/2]$, respectively.

Finally, we have that
\begin{equation}\begin{split}
M_{3}(h)&=\displaystyle\sum^{n}_{i=0}c_{i0}\int_{x_{1}}^{x_{2}}x^{i-2}\left(\dfrac{1}{((hx-b)/2-\sqrt{\Delta
})^2}-\dfrac{1}{((hx-b)/2+\sqrt{\Delta})^2}\right)\mathrm{d}x\\[-3pt]
&=\displaystyle\sum^{n}_{i=0}c_{i0}\int_{x_{1}}^{x_{2}}\dfrac{2x^{i-2}(hx-b)\sqrt{\Delta
}}{(x^{2}+bx+1)^2}\mathrm{d}x\\[-3pt]
&=\mu(h)S_{2}(h)+\nu(h)S_{3}(h)+\displaystyle\sum^{3}_{i=0}\bar{c}_{i0}R_{i}(h)+\sum^{n}_{i=5}\bar{c}_{i0}\omega_{i}(h)J_{i-3}(h),
\end{split}\end{equation}
where
$\bar{c}_{i0}, i=0,1,2,3,5,...,n,$ are
linear combinations of $c_{i0}, i=0,1,...,n$. $\mu(h)$, $\nu(h)$ and
$\omega_{i}(h)$ are polynomials of degree at most $1$, and
\begin{equation*}
R_{i}(h)=\displaystyle\int_{x_{1}}^{x_{2}}\dfrac{2x^{i-2}(hx-b)\sqrt{\Delta
}}{(x^{2}+bx+1)^2}\mathrm{d}x, \quad i=0,1,2,3.
\end{equation*}
A simple calculation shows that
\begin{equation}\begin{split}
R_{0}(h)&=-\dfrac{4(h-(2-b^2))(h-b^4+5b^2-2)\pi}{(4-b^2)^{3/2}},\\
R_{1}(h)&=\dfrac{2b(h-(2-b^2))(h-b^2+6)\pi}{(4-b^2)^{3/2}},\\
R_{2}(h)&=-\dfrac{4(h-(2-b^2))(h+2)\pi}{(4-b^2)^{3/2}},\\
R_{3}(h)&=-\dfrac{b(h+2)(b^2h-6h-2b^2+4)\pi}{(4-b^2)^{3/2}}-h\sqrt{4-h^2}\pi.
\end{split}\end{equation}
From (3.5), (3.9), (3.11) and (3.12), we obtain
\begin{equation}\begin{array}{ll}
M_{3}(h)=\bar{P}_{2}(h)+\displaystyle\dfrac{\bar{P}_{[(3n-7)/2]}(h)}{\sqrt{4-h^{2}}(2-h)^{n-4}(2+h)^{[(n-5)/2]}},
\end{array}\end{equation}
where $\bar{P}_{2}(h)$ and $\bar{P}_{[(3n-7)/2]}(h)$ are polynomials with $\deg \bar{P}_{2}\leq2$ and $\deg \bar{P}_{[(3n-7)/2]}\leq[(3n-7)/2]$, respectively.

\noindent{\bfseries{Proof of Theorem 2.}} It follows from Proposition 8, (3.2), (3.7), (3.10) and (3.13) that the Abelian integral is
\begin{equation}M(h)=\label{eq1}\left\{\begin{array}{ll}
Q_{2}(h)+\dfrac{Q_{[(3n-7)/2]}(h)}{\sqrt{4-h^{2}}(2-h)^{n-4}(2+h)^{[(n-5)/2]}} & \mbox{if $n\geq3$,}\\[2ex]
\widetilde{Q}_{2}(h) & \mbox{if $n\leq2$,}
\end{array} \right.
\end{equation}
where $Q_{2}(h)$, $\widetilde{Q}_{2}(h)$ and $Q_{[(3n-7)/2]}(h)$ denote polynomials with $\deg Q_{2}\leq{2}$, $\deg \widetilde{Q}_{2}\leq{2}$
and $\deg Q_{[(3n-7)/2]}\leq[(3n-7)/2]$, respectively.

If $n\leq2$, then the proof is similar to the proof of Theorem 1 and hence is omitted.

If $n\geq3$, in order to get the estimation of the number of zeros of $M(h)$, we take derivatives of $M(h)$ three
times. Noting that $M(2-b^2)=0$, it follows that
\begin{equation*}
H_{X_{2.10}}(n)\leq[\frac{3n-7}{2}]+3-1=[\frac{3n-3}{2}].
\end{equation*}

The proof is finished.

\section{$H_{X_{2.9}}(3)=3$ and $H_{X_{2.10}}(3)=3$}
In this section, we will study the sharp upper bound for the number of zeros of the Abelian integrals
with respect to the two systems for $n=3$, i.e., the determination of $H_{X_{2.9}}(3)$ and $H_{X_{2.10}}(3)$. For convenience,
we will denote
\[U^-=(-2,h_-), \quad U^+=(h_+,2).
\]

The exact expression of Abelian integral associated to system $X_{2.9}$ is given as follows:
\begin{equation}\begin{split}
M^{\pm}(h)&=\displaystyle\sum_{i+j=0}^{3}(-a_{ji}\bar{I}_{ij}(h)-b_{ij}I_{ij}(h))\\
&=a_{1}^{\pm}+a_{2}^{\pm}h+a_{3}^{\pm}h^2+a_{4}^{\pm}\sqrt{4-h^2}+a_{5}^{\pm}h\sqrt{4-h^2}, \quad h\in U^{\pm},
\end{split}\end{equation}
where
\begin{equation*}\begin{split}
a_{2}^{\pm}=&\displaystyle b_{00}\left(\pm\dfrac{2bc(6-b^2)\pi}{(4-b^2)^{3/2}}+
\dfrac{2(-2+c^2)\pi}{\sqrt{4-c^2}}\right)
+a_{00}\left(-\dfrac{2bc(-6+c^2)\pi}{(4-c^2)^{3/2}}\pm
\dfrac{2(-2+b^2)\pi}{\sqrt{4-b^2}}\right)\\
&-(b_{01}+a_{10})\left(\pm\dfrac{b\pi}{\sqrt{4-b^2}}+\dfrac{c\pi}{\sqrt{4-c^2}}\right)\pm\dfrac{2b_{11}\pi}{\sqrt{4-b^2}}+\dfrac{2a_{11}\pi}{\sqrt{4-c^2}}
\pm\dfrac{b(-b_{21}+a_{30})\pi}{\sqrt{4-b^2}}\\
&\pm\dfrac{4c(-2b_{10}+b_{20}b-2b_{30})\pi}{(4-b^2)^{3/2}}-\dfrac{4b(2a_{01}-a_{02}c+2a_{03})\pi}{(4-c^2)^{3/2}}
-\dfrac{c(a_{12}-b_{03})\pi}{\sqrt{4-c^2}},\\
a_{3}^{\pm}=&\displaystyle\pm\dfrac{4(b_{00}+b_{20})\pi}{(4-b^2)^{3/2}}+
\dfrac{4(a_{00}+a_{02})\pi}{(4-c^2)^{3/2}}\mp\dfrac{2b_{10}b\pi}{(4-b^2)^{3/2}}
-\dfrac{2a_{01}c\pi}{(4-c^2)^{3/2}}\\
&\pm\dfrac{b_{30}b(-6+b^2)\pi}{(4-b^2)^{3/2}}+\dfrac{a_{03}c(-6+c^2)\pi}{(4-c^2)^{3/2}},\\
a_{4}^{\pm}=&\pm\left(b_{21}-b_{03}+a_{12}-a_{30}\right)\pi,\\
a_{5}^{\pm}=&\pm(b_{30}+a_{03})\pi,
\end{split}\end{equation*}
and
\begin{equation}
a_{1}^{\pm}=m^{\pm}a_{2}^{\pm}+n^{\pm}a_{3}^{\pm}+p^{\pm}a_{4}^{\pm}+q^{\pm}a_{5}^{\pm}
\end{equation}
with
\begin{equation*}
m^{\pm}=-h_{\pm},\quad n^{\pm}=-h_{\pm}^{2},\quad
p^{\pm}=-\sqrt{4-h_{\pm}^{2}},\quad
q^{\pm}=-h_{\pm}\sqrt{4-h_{\pm}^{2}}.
\end{equation*}
Since $0\leq b<c<2$,
\begin{equation*}
\dfrac{\partial(a_{2}^{\pm},a_{3}^{\pm},a_{4}^{\pm},a_{5}^{\pm})}{\partial(a_{11},a_{01},a_{12},a_{03})}=-\dfrac{4c\pi^4}{(4-c^2)^2}\neq0,
\end{equation*}
which means that $a_{2}^{\pm},a_{3}^{\pm},a_{4}^{\pm},a_{5}^{\pm}$ are independent.

Denote
\[f_{0}^{\pm}(h)=h+m^{\pm},f_{1}^{\pm}(h)=h^2+n^{\pm},f_{2}^{\pm}(h)=\sqrt{4-h^2}+p^{\pm},f_{3}^{\pm}(h)=h\sqrt{4-h^2}+q^{\pm}.
\]
It follows from (4.1) and (4.2) that
\begin{equation}
M^{\pm}(h)=a_{2}^{\pm}f_{0}^{\pm}(h)+a_{3}^{\pm}f_{1}^{\pm}(h)+a_{4}^{\pm}f_{2}^{\pm}(h)+a_{5}^{\pm}f_{3}^{\pm}(h).
\end{equation}

By Theorem 1, we know that $H_{X_{2.9}}(3)\leq3$.
Now, we use \emph{Chebyshev
criterion} to show that $H_{X_{2.9}}(3)=3$, i.e., there exist the parameters $a_{i}^{\pm},
i=2,3,4,5$, such that $M^{\pm}(h)$ has exactly three zeros on
$U^{\pm}$, respectively. We introduce the following
definitions (see for instance \cite{MV}).

Let $f_{0},f_{1},...,f_{n-1}$ be analytic functions on an open
interval $L$ of $\mathbb{R}$. $(f_{0},f_{1},...,f_{n-1})$ is a
\emph{Chebyshev system} on $L$ if any nontrivial linear combination
\begin{equation*}
\lambda_{0}f_{0}(x)+\lambda_{1}f_{1}(x)+...+\lambda_{n-1}f_{n-1}(x)
\end{equation*}
has at most $n-1$ isolated zeros on $L$.

An ordered set $(f_{0},f_{1},...,f_{n-1})$ is a \emph{complete
Chebyshev system} on $L$ if
$(f_{0},f_{1},...,f_{i-1})$ is a Chebyshev system on $L$ for all
$i=1,2,...,n$.

An ordered set $(f_{0},f_{1},...,f_{n-1})$ is an \emph{extended
complete Chebyshev system} (in short, ECT-system) on $L$ if, for all
$i=1,2,...,n$, any nontrivial linear combination
\begin{equation}
\lambda_{0}f_{0}(x)+\lambda_{1}f_{1}(x)+...+\lambda_{i-1}f_{i-1}(x)
\end{equation}
has at most $i-1$ isolated zeros on $L$ counted with multiplicities.

\begin{remark}
If $(f_{0},f_{1},...,f_{n-1})$ is an ECT-system on $L$, then for
each $i=1,2,...,n$, there exists a linear of combination (4.4) with
exactly $i-1$ simple zeros on $L$ (see for instance Remark 3.7 in \cite{GV}).
\end{remark}

\begin{lemma}
\emph{(see \cite{MV})} $(f_{0},f_{1},...,f_{n-1})$ is an ECT-system on L
if, and only if, for each $i=1,2,...,n$,
\begin{equation*}
\Omega_{i}(x)=
\begin{vmatrix}
f_{0}(x) &  f_{1}(x) & \cdots &  f_{i-1}(x)\\
f_{0}^{\prime}(x) &f_{1}^{\prime}(x)  & \cdots & f_{i-1}^{\prime}(x) \\
\vdots & \vdots & \ddots & \vdots \\
f_{0}^{(i-1)}(x) & f_{1}^{(i-1)}(x) & \cdots & f_{i-1}^{(i-1)}(x) \\
  \end{vmatrix}\neq{0}
\end{equation*}
for all $x\in L$.
\end{lemma}

Simple computations show that, if $h\in U^{\pm}$,
\begin{equation*}\begin{array}{ll}
\Omega_{1}^{\pm}(h)=h-h_{\pm},\\[2ex]
\Omega_{2}^{\pm}(h)=(h-h_{\pm})^2,\\[2ex]
\Omega_{3}^{\pm}(h)=\displaystyle\dfrac{2}{(4-h^2)^{3/2}}
\left(h_{\pm}h^3-6h^2+16-2h^2_{\pm}\right)-2\sqrt{4-h^2_{\pm}},\\[2ex]
\Omega_{4}^{\pm}(h)=-\displaystyle\dfrac{24}{(4-h^2)^{3}}
\left((h_{\pm}^2-2)h^2-4h_{\pm}h+16-2h^2_{\pm}+\sqrt{4-h^2_{\pm}}\sqrt{4-h^2}(-4+h_{\pm}h)\right).
\end{array}\end{equation*}

Obviously, $\Omega_{1}^{\pm}(h)$ and $\Omega_{2}^{\pm}(h)$ do not vanish when $h\in U^{\pm}$. By direct computation, we obtain that $\Omega_{4}^{\pm}(h)$ has
only one zero $h=h_{\pm}$ with multiplicity four, which implies that
$\Omega_{4}^{\pm}(h)$ also does not change sign in each interval. However, it is a
little difficult to judge the sign of $\Omega_{3}^{\pm}(h)$.

\begin{proposition}
$(f_{0}^-(h),f_{1}^-(h),f_{2}^-(h),f_{3}^-(h))$ is an
ECT-system on $U^-$. If
$b^2+c^2\leq 4$, $(f_{0}^+(h),f_{1}^+(h),f_{2}^+(h),f_{3}^+(h))$ is an
ECT-system on $U^+$; if $b^2+c^2>4$, there is a real number
$d\in U^+$ such that $(f_{0}^+(h),f_{1}^+(h),f_{2}^+(h),f_{3}^+(h))$ is
an ECT-system on $(h_{+},d)$.
\end{proposition}
\begin{proof}
It is easy to verify that $\Omega_3^{\pm}(h_{\pm})=0$ and
\begin{equation}\label{f3}
{\Omega_3^{\pm}}^{\prime}(h)=-\displaystyle\dfrac{12h(h-h_{\pm})^2}{(4-h^2)^{5/2}}.
\end{equation}
Thus ${\Omega_3^{\pm}}^{\prime}(h)$ has zeros $h=0$ and $h=h_{\pm}$ with
multiplicity one and two, respectively.

Note that $U^-=(-2,h_-)$ and $h_{-}<0$. If $h\in U^-$, by \eqref{f3}, we find that ${\Omega_3^{-}}^{\prime}(h)>0$ holds, which implies that
$\Omega_{3}^-(h)<\Omega_3^-(h_{-})=0$. Therefore combining with the discussions above and lemma 10,
$(f_{0}^-(h),f_{1}^-(h),f_{2}^-(h),f_{3}^-(h))$ is an ECT-system on
$U^-$.

We split the proof into two cases if $h\in U^+$, i.e., $h\in(h_{+},2)$.

\emph{Case 1}. $h_{+}\geq0$, i.e., $b^2+c^2\leq4$.

By \eqref{f3}, ${\Omega_3^+}^{\prime}(h)<0$ holds on $U^+$. It follows that
$\Omega_{3}^+(h)<\Omega_3^+(h_{+})=0$. Thus $\Omega_{3}^+(h)$ does not vanish on $U^+$, which shows that
$(f_{0}^+(h),f_{1}^+(h),f_{2}^+(h),f_{3}^+(h))$ is an ECT-system on
$U^+$.

\emph{Case 2}. $h_{+}<0$, i.e., $b^2+c^2>4$.

We get ${\Omega_3^+}^{\prime}(h)>0$ for $h\in(h_{+},0)$ and ${\Omega_3^+}^{\prime}(h)<0$ for $h\in(0,2)$ by \eqref{f3}. Noting that $\Omega_3^+(h_{+})=0$,
$\Omega_3^+(0)=(\small\sqrt{4-h_{+}^2}-2)^2/2>0$ and $\Omega_{3}^+(h)\rightarrow{-\infty}$ as $h\rightarrow2^{-}$, we know that there is a simple zero of $\Omega_{3}^+(h)$
on $(0,2)$. Thus $(f_{0}^+(h),f_{1}^+(h),f_{2}^+(h),f_{3}^+(h))$ is not
an ECT-system on $U^+$ by lemma 10. But there is a real number
$d\in U^+$ such that $\Omega_{3}^+(h)>0$ for $h\in(h_{+},d)$. Consequently,
$(f_{0}^+(h),f_{1}^+(h),f_{2}^+(h),f_{3}^+(h))$ is an ECT-system on
$(h_{+},d)$.

The proof is finished.
\end{proof}

Next, let us consider system $X_{2.10}$.

Similarly, the Abelian integral associated to system $X_{2.10}$ has
the form:
\begin{equation}\begin{split}
M(h)&=\displaystyle\sum_{i+j=0}^{3}(-a_{ji}-b_{ij})I_{ij}(h)\\
&=a_{1}+a_{2}h+a_{3}h^2+a_{4}\sqrt{4-h^2}+a_{5}h\sqrt{4-h^2}, \quad h\in(2-b^2,2),
\end{split}\end{equation}
where
\begin{equation*}\begin{split}
a_{2}=&-\displaystyle\dfrac{4(b_{00}+a_{00})(4-6b^2+b^4)\pi}{(4-b^2)^{3/2}}
-\dfrac{8(b_{10}+a_{01}+b_{30}+a_{03})b\pi}{(4-b^2)^{3/2}}
-\dfrac{2(b_{01}+a_{10})b\pi}{\sqrt{4-b^2}}\\
&+\dfrac{4(b_{20}+a_{02})b^2\pi}{(4-b^2)^{3/2}}+\dfrac{2(b_{11}+a_{11})\pi}{\sqrt{4-b^2}}
-\dfrac{(b_{21}+a_{12}-b_{03}-a_{30})b\pi}{\sqrt{4-b^2}},
\end{split}\end{equation*}
\begin{equation*}\begin{split}
a_{3}&=\displaystyle\dfrac{4(b_{00}+b_{20}+a_{00}+a_{02})\pi}{(4-b^2)^{3/2}}-\dfrac{2(b_{10}+a_{01})b\pi}{(4-b^2)^{3/2}}
+\dfrac{(b_{30}+a_{03})b(-6+b^2)\pi}{(4-b^2)^{3/2}},\\
a_{4}&=(b_{21}-b_{03}+a_{12}-a_{30})\pi,\\
a_{5}&=(b_{30}+a_{03})\pi,
\end{split}\end{equation*}
and
\begin{equation}\begin{split}
a_{1}=ma_{2}+na_{3}+pa_{4}+qa_{5},
\end{split}\end{equation}
with
\begin{equation*}
m=-2+b^2, \quad n=-(-2+b^2)^2,\quad
p=-b\sqrt{4-b^2}, \quad q=b\sqrt{4-b^2}(-2+b^2).
\end{equation*}
It follows from
\begin{equation*}
\dfrac{\partial(a_{2},a_{3},a_{4},a_{5})}{\partial(a_{11},a_{01},a_{12},a_{03})}=
-\dfrac{4b\pi^4}{(4-b^2)^2}\neq0
\end{equation*}
that $a_{2},a_{3},a_{4},a_{5}$ are independent.

Denote

$f_{0}(h)=h+m, f_{1}(h)=h^2+n, f_{2}(h)=\sqrt{4-h^2}+p, f_{3}(h)=h\sqrt{4-h^2}+q$.\\
By (4.6) and (4.7),
\begin{equation}
M(h)=a_{2} f_{0}(h)+a_{3}f_{1}(h)+a_{4} f_{2}(h)+a_{5} f_{3}(h).
\end{equation}

\begin{proposition}
When $0<b\leq\sqrt{2}$, $( f_{0}(h), f_{1}(h), f_{2}(h), f_{3}(h))$ is an
ECT-system on $(2-b^2,2)$; when $\sqrt{2}<b<2$, there is a real number
$d\in(2-b^2,2)$ such that $( f_{0}(h), f_{1}(h), f_{2}(h), f_{3}(h))$ is
an ECT-system on $(2-b^2,d)$.
\end{proposition}
\begin{proof}
By direct computations,
\begin{equation*}\begin{split}
\Omega_{1}(h)=&h+b^2-2,\\
\Omega_{2}(h)=&(h+b^2-2)^2,\\
\Omega_{3}(h)=&\displaystyle\dfrac{2}{(4-h^2)^{3/2}}
\left((2-b^2)h^3-6h^2-2(-4-4b^2+b^4)\right)-2b\sqrt{4-b^2},\\
\Omega_{4}(h)=&-\displaystyle\dfrac{24}{(4-h^2)^{3}}
\left((2-4b^2+b^4)h^2-4(2-b^2)h+8+8b^2-2b^4\right.\\
&\left.+b\sqrt{4-b^2}\sqrt{4-h^2}(-4+(2-b^2)h)\right),
\end{split}\end{equation*}
where $h\in(2-b^2,2)$.

Obviously, $\Omega_{1}(h)$ and $\Omega_{2}(h)$ do not vanish on $(2-b^2,2)$. It is easy to verify that $h=2-b^2$
is a unique zero of $\Omega_{4}(h)$ with multiplicity four, which implies that $\Omega_{4}(h)\neq0$ for $h\in(2-b^2,2)$.
Therefore we just need to determine the sign of $\Omega_{3}(h)$. The discussion is similar to the proof of Proposition 11 for $h\in U^+$, thus we omit it.
\end{proof}
\noindent{\bfseries{Proof of Theorem 4.}} It follows from Theorem 1
that $H_{X_{2.9}}(3)\leq3$. By Proposition 11 and
Remark 9, we know that there exist $a_{i}^{\pm}, i=2,3,4,5$, such that the Abelian integral
$M^{\pm}(h)$ has exactly three zeros on $U^{\pm}$,
respectively. Therefore $H_{X_{2.9}}(3)=3$.

Similarly, $H_{X_{2.10}}(3)=3$ by Theorem 2, Proposition 12 and
Remark 9.
\section{Simultaneous bifurcation and distribution of limit cycles for system $X_{2.9}$}
In what follows we will consider the simultaneous bifurcation of limit cycles
bifurcating from the two period annuli of system $X_{2.9}$ under cubic polynomial perturbations. Note that $a_4^-=-a_4^+$ and $a_5^-=-a_5^+$.
Rewrite (4.3) as
\begin{equation}
M^{\pm}(h)=a_{2}^{\pm}f_{0}^{\pm}(h)+a_{3}^{\pm}f_{1}^{\pm}(h)\pm a_{4}^{+}f_{2}^{\pm}(h)\pm a_{5}^+f_{3}^{\pm}(h),\quad h\in U^{\pm}.
\end{equation}

\noindent{\bfseries{Proof of Theorem 5.}} We study the number of zeros of the two Abelian integrals simultaneously. Note that
\[
\dfrac{\partial(a_{2}^+,a_{2}^-,
a_{3}^+,a_{3}^-,a_{4}^+,a_{5}^+)}{\partial(a_{11},b_{11},b_{00},a_{01},b_{21},b_{30})}=\dfrac{128c{\pi}^6}{(4-b^2)^2(4-c^2)^2}\neq0.
\]
Thus, we can consider $a_{2}^+,a_{2}^-,a_{3}^+,a_{3}^-,a_{4}^+,a_{5}^+$ to be independent.

Firstly, let $a_{3}^+=a_{3}^-=a_{4}^+=a_{5}^+=0$ and $a_{2}^+, a_{2}^-\neq0$, then neither $M^+(h)$ has zeros on $U^+$ nor
$M^-(h)$ has zeros on
$U^-$. Hence, the distribution (0,0) is
possible.

Secondly, let $a_{4}^+=a_{5}^+=0$. By Proposition 11 and Remark 9, we can choose $a_{2}^+,
a_{3}^+$ (resp. $a_{2}^-, a_{3}^-$) such that $M^+(h)$
(resp. $M^-(h)$) has none or one zero on $U^+$
(resp. $U^-$). Thus the distributions (1,0), (0,1) and (1,1) can be
achieved.

Thirdly, take $a_{5}^+=0$. It follows from Proposition 11 and Remark 9 that there exist
$\rho_{1}, \rho_{2}, \rho_{3}(\neq0)$ and $\sigma_{1}, \sigma_{2}, \sigma_{3}(\neq0)$ such that
\[
\rho_{1}f_{0}^{+}(h)+\rho_{2}f_{1}^{+}(h)+\rho_{3}f_{2}^{+}(h)\quad \mbox{and} \quad
\sigma_{1}f_{0}^{-}(h)+\sigma_{2}f_{1}^{-}(h)-\sigma_{3}f_{2}^{-}(h)
\]
have $u$, $v$ zeros respectively, with $0\leq u, v\leq2$. Then multiplying by
$\rho_{3}/\sigma_{3}$ the
second function, it turns out that
\[
\sigma_{1}\rho_{3}/\sigma_{3}f_{0}^{-}(h)
+\sigma_{2}\rho_{3}/\sigma_{3}f_{1}^{-}(h)-
\rho_{3}f_{2}^{-}(h)
\]
has $v$ zeros.
Choose $a_{2}^+,
 a_{2}^-,a_{3}^+, a_{3}^-,a_{4}^+$ such that
\begin{equation*}
a_{2}^+=\rho_{1}, \quad
a_{2}^-=\sigma_{1}\rho_{3}/\sigma_{3},\quad
a_{3}^+=\rho_{2},\quad
a_{3}^-=\sigma_{2}\rho_{3}/\sigma_{3},\quad
a_{4}^+=\rho_{3}.
\end{equation*}
It follows that the distribution $(u,v), 0\leq u, v\leq2$ is
possible.

Finally, suppose that $a_{5}^+\neq0$. It is easy to obtain that
\[
{M^{+}}^{(3)}(h)=-\dfrac{12(a_4^+h+4a_5^+)}{(4-h^2)^{5/2}} \quad \mbox{and} \quad {M^{-}}^{(3)}(h)=\dfrac{12(a_4^+h+4a_5^+)}{(4-h^2)^{5/2}}.
\]
If $a_{4}^+=0$, then $M^+(h)$ and $M^-(h)$ have at most two zeros on $U^+$ and $U^-$, respectively. Thus, in order to get more limit cycles, $a_{4}^+\neq0$.
We find that ${M^{+}}^{(3)}(h)$ and ${M^{-}}^{(3)}(h)$ have the same unique zero $h_0=-4a_5^+/a_4^+$, which implies that  ${M^{+}}^{(2)}(h)$ has at most two (resp. one) zero(s) on $U^+$ and ${M^{-}}^{(2)}(h)$ has at most one (resp. two) zero(s) on $U^-$ if $h_0\in U^+$ (resp. $h_0\in U^-$), otherwise, both ${M^{+}}^{(2)}(h)$ and ${M^{-}}^{(2)}(h)$ have at most one zero on $U^+$ and $U^-$, respectively. Thus, $M^{+}(h)$ has at most three (resp. two) zeros on $U^+$ and $M^{-}(h)$ has at most two (resp. three) zeros on $U^-$, which means the distribution (3,3) is impossible. To show (3,2) is achievable, we give the following asymptotic expansions of $M^{\pm}(h)$
at $h=h_{\pm}$, respectively:
\begin{equation}\begin{split}
M^{\pm}(h)=s_1^{\pm}(h-h_{\pm})+s_2^{\pm}(h-h_{\pm})^2+s_3^{\pm}(h-h_{\pm})^3+s_4^{\pm}(h-h_{\pm})^4+\cdots,
\end{split}\end{equation}
where
\begin{equation*}\begin{split}
s_1^{\pm}&=a_2^{\pm}+2a_3^{\pm}h_{\pm}\mp\frac{a_4^+h_{\pm}+2a_5^+(h_{\pm}^2-2)}{(4-h_{\pm}^2)^{1/2}}, \quad
s_2^{\pm}=a_3^{\pm}\mp\frac{2a_4^+-a_5^+h_{\pm}(h_{\pm}^2-6)}{(4-h_{\pm}^2)^{3/2}},\\
s_3^{\pm}&=\mp\dfrac{2(a_4^+h_{\pm}+4a_5^+)}{(4-h_{\pm}^2)^{5/2}},\quad \quad \quad \quad \quad \quad \quad \quad \quad \,
s_4^{\pm}=\mp\dfrac{2(a_4^++5a_5^+h_{\pm}+a_4^{+}h_{\pm}^2)}{(4-h_{\pm}^2)^{7/2}}.
\end{split}\end{equation*}
Since
\[
\dfrac{\partial(s_{1}^+,s_{1}^-,
s_{2}^+,s_{2}^-,s_{3}^+,s_{4}^+)}{\partial(a_{2}^+,a_{2}^-,
a_{3}^+,a_{3}^-,a_{4}^+,a_{5}^+)}=-\dfrac{4}{(4-h_{+}^2)^5}\neq0,
\]
we consider $s_{1}^+,s_{1}^-,
s_{2}^+,s_{2}^-,s_{3}^+,s_{4}^+$ as the new independent parameters. Denote
\[
M^+(h)=M^+(h,s_{1}^+,s_{2}^+,s_{3}^+,s_{4}^+) \quad \mbox{and} \quad M^-(h)=M^-(h,s_{1}^-,s_{2}^-,s_{3}^-,s_{4}^-).
\]
Without loss of
generality, suppose $s_{4}^+>0$. To get more zeros of $M^+(h)$, we choose $s_{i}^+$ and $h_{i}\in U^+$, $i=
4, 3, 2, 1$, such that
\begin{equation}\begin{split}\label{M^+}
&M^+(h_4, 0, 0, 0, s_{4}^+)>0, \quad \quad \,M^+(h_3,0,0,s_{3}^+,s_{4}^+)<0,\\
&M^+(h_2,0,s_{2}^+,s_{3}^+,s_{4}^+)>0, \quad M^+(h_1,s_{1}^+,s_{2}^+,s_{3}^+,s_{4}^+)<0,
\end{split}\end{equation}
and $0 < |s_{1}^+|\ll|s_{2}^+|\ll |s_{3}^+|\ll |s_{4}^+|$, $h_+ < h_1 < h_2< h_3 < h_4 < 2$. It is
easy to show that $M^+(h)$ has three zeros which tend to $h_+$. Once $s_{3}^+,s_{4}^+$ are chosen, the sign of $s_3^-$ is
determined. For example, from the analysis above, we know that $s_{4}^+>0$ and $s_{3}^+<0$ by \eqref{M^+}. The result that $M^+(h)$ has three zeros on $U^+$ implies $h_0=-4a_5^+/a_4^+\in U^+$. Thus $a_4^+<0$ and $s_{3}^->0$.
There exists $h_5\in U^-$ such that $M^-(h_5,0,0,s_{3}^-,s_{4}^-)<0$. Take $s_{i}^-, i=2,1$ and
$h_{i}\in U^-, i=6,7$, such that
\[
M^-(h_6,0,s_{2}^-,s_{3}^-,s_{4}^-)>0, \quad M^-(h_7,s_{1}^-,s_{2}^-,s_{3}^-,s_{4}^-)<0,
\]
and $0 < |s_{1}^-|\ll|s_{2}^-|\ll \mbox{min}\{|s_{3}^-|,|s_{4}^-|\}$, $-2 < h_5 < h_6< h_7 < h_-$. It follows that
$M^-(h)$ has at least two zeros on $U^-$, which tend to $h_-$. Thus, (3,2) is realizable. Similarly, other configurations of limit cycles $(u,v)$ with $0\leq u, v\leq 3, u+v\leq5$
can be realized in this way.

This completes the proof.
\section*{Acknowledgements}
Research is supported by the International Program of Project 985, Sun Yat-Sen University, the NSF of China (No. 11171355) and the Ph.D. Programs Foundation of Ministry of Education of China (No. 20100171110040).


\begin{thebibliography}{99}
\bibitem{AZ}
{\sc R. Asheghi and H. R. Z. Zangeneh}: {\it Bifurcations and distribution of limit cycles which appear from two nests of periodic orbits}, Nonlinear Anal., {\bf 73}(2010), 2398-2409.

\bibitem{BL1}
{\sc A. Buic\u{a} and J. Llibre}: {\it Limit cycles of a perturbed cubic polynomial
differential center}, Chaos, Solitons and Fractals, {\bf 32}(2007), 1059-1069.

\bibitem{CL}
{\sc L. Cair\'{o} and J. Llibre}: {\it Phase portraits of cubic polynomial vector fields of Lotka-Volterra type having a rational first integral of degree 2}, J. Phys. A: Math. Theor., {\bf 40}(2007), 6329-6348.


\bibitem{CLY}
{\sc S. Chow, C. Li and Y. Yi}: {\it The cyclicity of period annuli of degenerate quadratic Hamiltonian systems with elliptic segment loops}, Ergodic Theory Dynam. Systems, {\bf 22}(2002), 349-374.


\bibitem{CGP}
{\sc B. Coll, A. Gasull and R. Prohens}: {\it Bifurcation of limit cycles
from two families of centers}, Dyn. Contin. Discrete Impuls. Syst.,
Ser. A, Math. Anal., {\bf 12}(2005), 275-287.

\bibitem{CLP}
{\sc B. Coll, C. Li and R. Prohens}: {\it Quadratic perturbations of a class of quadratic reversible systems with two centers}, Discrete Contin. Dyn. Syst.,
{\bf 24(3)}(2009), 699-729.

\bibitem{GGJ}
{\sc A. Garijo, A. Gasull and X. Jarque}: {\it Simultaneous bifurcation of limit cycles from two nests of periodic orbits}, J. Math. Anal. Appl., {\bf 341}(2008), 813-824.

\bibitem{GLV}
{\sc H. Giacomini, J. Llibre and M. Viano}: {\it On the shape of limit cycles that bifurcate from Hamiltonian centers}, Nonlinear Analysis: Theory, Methods \& Applications, {\bf 41}(2000), 523-537.

\bibitem{GV}
{\sc M. Grau and J. Villadelprat}: {\it Bifurcation of critical periods from Pleshkans
isochrones}, J. London Math. Soc., {\bf 81}(2010), 142-160.

\bibitem{HI}
{\sc E. Horozov and I. D. Iliev}: {\it Linear estimate for the number of zeros
of Abelian integrals with cubic Hamiltonians}, Nonlinerarity, {\bf 11}(1998), 1521-1537.

\bibitem{I}
{\sc I. D. Iliev}: {\it Perturbations of quadratic centers},  Bull. Sci. Math., {\bf 122}(1998), 107-161.

\bibitem{LLLZ}
{\sc C. Li, W. Li, J. Llibre and Z. Zhang}: {\it On the limit cycles of polynomial differential systems with homogeneous nonlinearities}, Proc. Edinburgh Math. Soc., {\bf 43}(2000), 529-543.

\bibitem{LLLZ1}
{\sc C. Li, W. Li, J. Llibre and Z. Zhang}: {\it Linear estimate for
the number of zeros of Abelian integrals for quadratic isochronous
centres}, Nonlinearity, {\bf 13}(2000), 1775-1800.

\bibitem{LLLZ2}
{\sc C. Li, W. Li, J. Llibre and Z. Zhang}: {\it Linear estimate of
the number of zeros of Abelian integrals for some cubic isochronous
centres}, J. Differential Equations, {\bf 180}(2002), 307-333.

\bibitem{L}
{\sc C. Liu}: {\it The cyclicity of period annuli of a class of quadratic reversible systems with two centers}, J. Differential Equations, {\bf 252(10)}(2012), 5260-5273.

\bibitem{LLM}
{\sc J. Llibre, B. D. Lopes  and J. R. De Moraes}: {\it Limit cycles of cubic polynomial differential systems with rational first integrals of degree 2}, submitted for publication.

\bibitem{MV}
{\sc F. Ma\~{n}osas and J. Villadelprat}: {\it Bounding the number of zeros of
certain Abelian integrals}, J. Differential Equations, {\bf 251}(2011), 1656-1669.

\bibitem{YH}
{\sc H. Yao and M. Han}: {\it The number of limit cycles of a class of polynomial differential systems}, Nonlinear Analysis: Theory, Methods \& Applications, {\bf 75(1)}(2012), 341-357.

\bibitem{Z}
{\sc Y. Zhao}: {\it On the number of limit cycles in quadratic
perturbations of quadratic codimension-four centres},
Nonlinearity, {\bf 24}(2011), 2505-2522.

\bibitem{ZZ}
{\sc Y. Zhao and Z. Zhang}: {\it Linear estimate of the number of zeros
of Abelian integrals for a kind of quartic Hamiltonians},
J. Differential Equations, {\bf 155}(1999), 73-88.

\end{thebibliography}
\end{document}